\documentclass{amsart}
\usepackage{amsmath}
\usepackage{amssymb}
\usepackage{amsthm}
\usepackage{amsrefs}
\usepackage{graphicx}

\newtheorem{theorem}{Theorem}
\newtheorem{lemma}[theorem]{Lemma}
\newtheorem{corollary}[theorem]{Corollary}

\theoremstyle{remark}
\newtheorem{remark}[theorem]{Remark}
\theoremstyle{definition}
\newtheorem{definition}[theorem]{Definition}

\begin{document}

\title[Graph Cases of Riemannian PMT and Penrose Inequality]{The Graph Cases of the Riemannian Positive Mass and Penrose Inequalities in All Dimensions}
\date{October 20, 2010}
\author{Mau-Kwong George Lam}
\address{Mathematics Department\\
Duke University, Box 90320\\
Durham, NC 27708-0320}
\email{glam@math.duke.edu}

\begin{abstract}
We consider complete asymptotically flat Riemannian manifolds that are the graphs of smooth functions over $\mathbb R^n$. By recognizing the scalar curvature of such manifolds as a divergence, we express the ADM mass as an integral of the product of the scalar curvature and a nonnegative potential function, thus proving the Riemannian positive mass theorem in this case. If the graph has convex horizons, we also prove the Riemannian Penrose inequality by giving a lower bound to the boundary integrals using the Aleksandrov-Fenchel inequality.
\end{abstract}

\maketitle

\section{Introduction}

The Riemannian positive mass theorem states that an asymptotically flat Riemannian manifold $M^n$ with nonnegative scalar curvature has nonnegative ADM mass, and that the ADM mass is strictly positive unless $M^n$ is isometric to flat $\mathbb R^n$. The theorem was first proved in 1979 by Schoen and Yau \cites{schoen_yau1,schoen_yau2} for manifolds of dimension $n\leq 7$ using minimal surface techniques. Witten \cite{witten} later proved the theorem for spin manifolds of any dimension using spinors and the Dirac operator.

The Penrose inequality can be viewed as a generalization of the positive mass theorem in the presence of an area outer minimizing horizon. A horizon is simply a minimal surface, and we say that it is area outer minimizing if every other surface which encloses it has greater area. When an asymptotically flat Riemannian manifold $M^n$ with nonnegative scalar curvature contains an area outer minimizing horizon $\Sigma$, the Riemannian Penrose inequality gives a lower bound for the ADM mass $m$ in terms of the $(n-1)$ volume $A$ of $\Sigma$ and the volume $\omega_{n-1}$ of the unit $(n-1)$ sphere:
\begin{equation} \label{rpi}
	m \geq \frac 12 \left( \frac{A}{\omega_{n-1}} \right) ^{\frac{n-2}{n-1}},
\end{equation}
with equality if and only if $M^n$ is isometric to a Schwarzschild metric. This inequality was first proved in dimension $n=3$ by Huisken and Illmanen \cite{imcf} in 1997 for the case of a single horizon using inverse mean curvature flow, a method originally proposed by Geroch \cite{geroch}, Jang and Wald \cite{jang_wald}. In 1999, Bray \cite{bray_RPI} extended this result to the general case of a horizon with multiple components using a conformal flow of metrics. Note that in dimension three, the Riemannian Penrose inequality is $m \geq \sqrt{A/16\pi}$. Later, Bray and Lee \cite{bray_lee} generalized Bray's proof for dimensions $n\leq 7$, with the extra requirement that $M$ be spin for the rigidity statement. To the author's knowledge, there are no known proofs of the Riemannian Penrose inequality in dimensions $n\geq 8$ other than in the spherically symmetric cases.

The case of equality of the Riemannian Penrose inequality is attained by the Schwarzschild metric. It is conformal to $\mathbb R^n\backslash\{0\}$ and may be expressed as
\begin{equation*}
	\left( \mathbb R^n \backslash \{0\}, \left( 1 + \frac{m}{2|x|^{n-2}} \right) ^{4/(n-2)} \delta \right),
\end{equation*}
where $m$ is a positive constant and $\delta$ is the flat Euclidean metric. Moreover, when $n=3$, the Schwarzschild metric $(M^3,g) = (\mathbb R^3\backslash\{0\}, (1+m/2|x|)^4\delta)$ can also be isometrically embedded as a rotating parabola in $\mathbb R^4$ as the set of points $\{(x,y,z,w)\}\subset\mathbb R^4$ satisfying $|(x,y,z)| = \frac{w^2}{8m}+2m$:

\begin{figure}[h]
   \begin{center}
   \includegraphics[width=120mm]{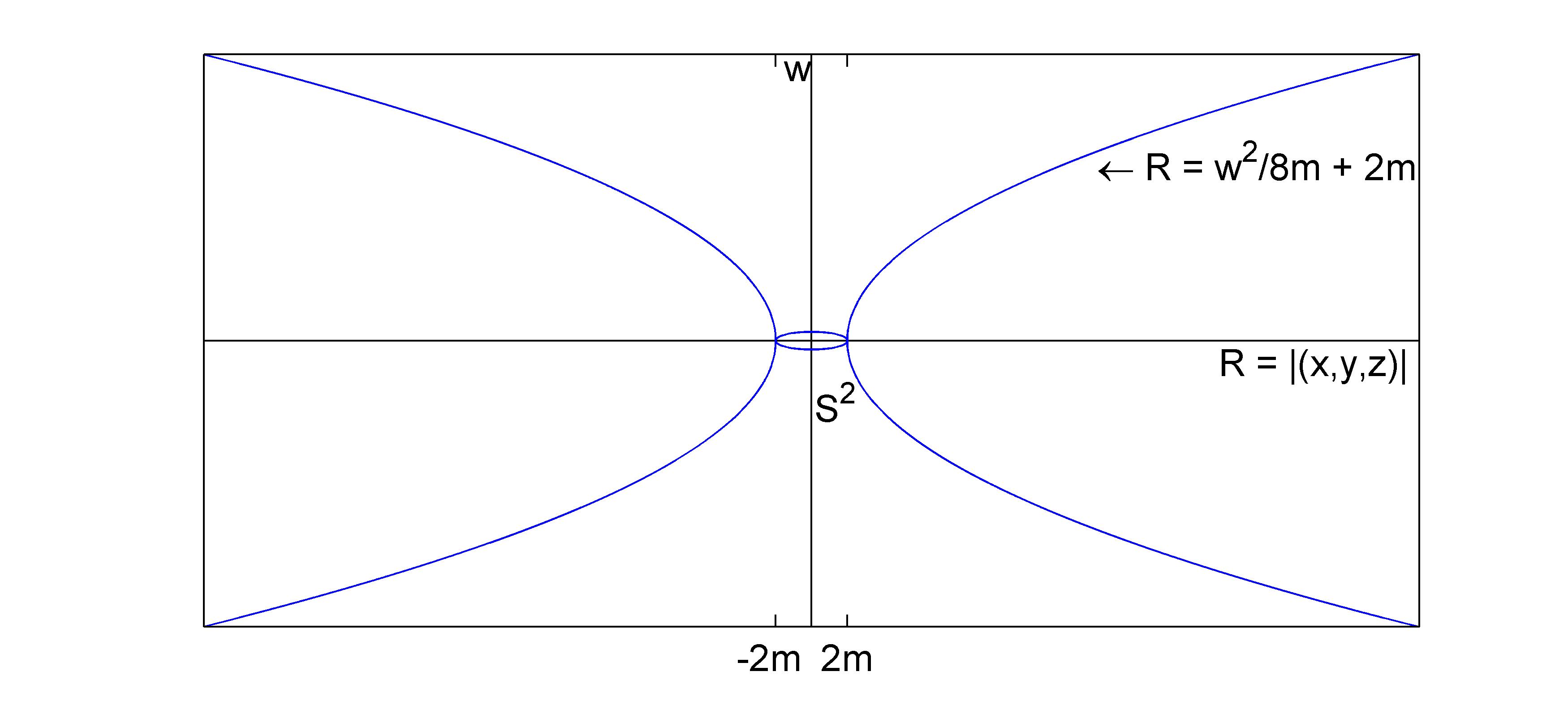}
   \end{center}
   \caption{The three dimensional Schwarzschild metric of mass $m>0$ (in blue) viewed as a spherically symmetric submanifold of four dimensional Euclidean space $\{(x,y,z,w)\}$ satisfying $R = w^2/8m + 2m$, where $R = \sqrt{x^2 + y^2 + z^2}$. Figure courtesy of Hubert Bray.\label{PositiveSchwarzschild}}
\end{figure}

Solving for $w$, we see that the end of the three dimensional Schwarzschild metric containing infinity, called the outer end, is the graph of the spherically symmetric function $f:\mathbb R^3\backslash B_{2m}(0)\to\mathbb R$ given by $f(r) = \sqrt{8m(r-2m)}$, where $r = |(x,y,z)|$. In this case, one can check directly that the ADM mass of $(M^3,g)$ is the positive constant $m$ by computing a certain boundary integral at infinity involving the function $f$.

That an end of the three dimensional Schwarzschild metric can be isometrically embedded in $\mathbb R^4$ as the graph of a function over $\mathbb R^3 \backslash B_{2m}(0)$ raises the following questions: if $\Omega$ is a bounded open set in $\mathbb R^n$ and $f$ is a smooth function on $\mathbb R^n \backslash \Omega$ such that the graph of $f$ is an asymptotically flat manifold $M$ with nonnegative scalar curvature $R$ and horizon $f(\partial\Omega)$, can we prove the Penrose inequality for $M$ using elementary techniques in this setting? And if so, do can get a stronger statement than the standard Penrose inequality? We answer both questions in the affirmative, and we will begin by proving a stronger version of the Riemannian positive mass theorem for manifolds that are graphs over $\mathbb R^n$ by expressing $R$ as a divergence and applying the divergence theorem, giving the ADM mass as an integral over the manifold of the product of $R$ and a nonnegative potential function. In the presence of a boundary whose connected components are convex, we prove a stronger Penrose inequality by giving lower bounds to the boundary integrals using the Aleksandrov-Fenchel inequality. Before we state our theorems, we begin with some definitions.

\begin{definition}\cite{schoen}
	A complete Riemannian manifold $(M^n,g)$ of dimension $n$ is said to be {\bf asymptotically flat} if there is a compact subset $K \subset M^n$ such that $M^n \backslash K$ is diffeomorphic to $\mathbb R^n \backslash \{|x|\leq 1\}$, and a diffeomorphism $\Phi:M^n\backslash K\to \mathbb R^n \backslash \{|x|\leq 1\}$ such that, in the coordinate chart defined by $\Phi$, $g=g_{ij}(x)dx^idx^j$, where
\begin{align*}
	g_{ij}(x)														 &= \delta_{ij}+O(|x|^{-p}) \\
	|x||g_{ij,k}(x)|+|x|^2|g_{ij,kl}(x)| &= O(|x|^{-p}) \\
	|R(g)(x)|														 &= O(|x|^{-q})
\end{align*}
for some $q>n$ and $p>(n-2)/2$.
\end{definition}

\begin{remark}
	Note that all this means is that outside a compact set, $M^n$ is diffeomorphic to $\mathbb R^n$ minus a closed ball and the metric $g$ decays sufficiently fast to the flat metric at infinity. The constants $p$ and $q$ are chosen so that the ADM mass (see below) is finite.
\end{remark}

For such an asymptotically flat manifold $(M^n,g)$, we can define its total mass (called the ADM mass):

\begin{definition} \cites{schoen}
	The {\bf ADM mass} $m$ of a complete, asymptotically flat manifold $(M^n,g)$ is defined to be
\begin{equation*}
	m = \lim_{r\to\infty} \frac{1}{2(n-1)\omega_{n-1}} \int_{S_r} \sum_{i,j} (g_{ij,i}-g_{ii,j}) \nu_j dS_r,
\end{equation*}
where $\omega_{n-1}$ is the volume of the $n-1$ unit sphere, $S_r$ is the coordinate sphere of radius $r$, $\nu$ is the outward unit normal to $S_r$ and $dS_r$ is the area element of $S_r$ in the coordinate chart.
\end{definition}
This definition in dimension three was originally due to Arnowitt, Deser and Misner \cite{adm}. Bartnik \cite{bartnik} showed that the ADM mass is independent of the choice of asymptotically flat coordinates.

Given a smooth function $f:\mathbb R^n \to \mathbb R$, the graph of $f$ is a complete Riemannian manifold. Since the graph of $f$ with the induced metric from $\mathbb R^{n+1}$ is isometric to $(M^n,g) = (\mathbb R^n, \delta + df\otimes df)$, we will from now on refer to $(M^n,g)$ as the graph of $f$. For such a graph, we can rephrase the notion of asymptotic flatness in terms of the function $f$:

\begin{definition}
	Let $f:\mathbb R^n \to \mathbb R$ be a smooth function and let $f_i$ denote the $i$th partial derivative of $f$. We say that $f$ is {\bf asymptotically flat} if
\begin{align*}
	f_i(x)													 &= O(|x|^{-p/2}) \\
	|x||f_{ij}(x)|+|x|^2|f_{ijk}(x)| &= O(|x|^{-p/2})
\end{align*}
at infinity for some $p > (n-2)/2$.
\end{definition}

We can now give the precise statement of our first theorem:

\begin{theorem}[Positive mass theorem for graphs over $\mathbb R^n$] \label{pmt}
Let $(M^n,g)$ be the graph of a smooth asymptotically flat function $f:\mathbb R^n \to \mathbb R$ with the induced metric from $\mathbb R^{n+1}$. Let $R$ be the scalar curvature and $m$ the ADM mass of $(M^n,g)$. Let $\nabla f$ denote the gradient of $f$ in the flat metric and $|\nabla f|$ its norm with respect to the flat metric. Let $dV_g$ denote the volume form on $(M^n,g)$. Then
\begin{equation*}
	m = \frac{1}{2(n-1)\omega_{n-1}} \int_{M^n} R \frac{1}{\sqrt{1+|\nabla f|^2}} dV_g.
\end{equation*}
In particular, $R \geq 0$ implies $m \geq 0$.
\end{theorem}

Now let $\Omega$ be a bounded open set in $\mathbb R^n$ with $\Sigma = \partial \Omega$ and $f$ a smooth function on $\mathbb R^n \backslash \Omega$. If $f(\Sigma)$ is contained in a level set of $f$, then the mean curvature $H$ of $f(\Sigma)$ in $(M^n,g)$ and the mean curvature $H_0$ with respect to the flat metric $\delta$ are related by
\begin{equation*}
	H = \frac{1}{\sqrt{1+|\nabla f|^2}} H_0.
\end{equation*}
Thus if $|\nabla f(x)| \to \infty$ as $x \to \Sigma$, then $f(\Sigma)$ is a horizon in $(M^n,g)$.

\begin{theorem} \label{pi}
Let $\Omega$ be a bounded and open (but not necessarily connected) set in $\mathbb R^n$ and $\Sigma = \partial \Omega$. Let $f: \mathbb R^n \backslash \Omega \to \mathbb R$ be a smooth asymptotically flat function such that each connected component of $f(\Sigma)$ is in a level set of $f$ and $|\nabla f(x)| \to \infty$ as $x \to \Sigma$. Let $(M^n,g)$ be the graph of $f$ with the induced metric from $\mathbb R^n \backslash \Omega \times \mathbb R$ and ADM mass $m$. Let $H_0$ be the mean curvature of $\Sigma$ in $(\mathbb R^n \backslash \Omega, \delta)$. Then
\begin{equation*}
	m = \frac{1}{2(n-1)\omega_{n-1}} \int_{\Sigma} H_0 d\Sigma + \frac{1}{2(n-1)\omega_{n-1}} \int_{M^n} R \frac{1}{\sqrt{1+|\nabla f|^2}} dV_g.
\end{equation*}
\end{theorem}

Let $\Omega_i$ be the connected components of $\Omega$, $i=1,\ldots,k$, and let $\Sigma_i = \partial \Omega_i$. If we assume that each $\Omega_i$ is convex, then we have the following Penrose inequality:
\begin{corollary}[Penrose inequality for graphs on $\mathbb R^n$ with convex boundaries] \label{pic}
With the same hypotheses as in Theorem \ref{pi}, and the additional assumption that each connected component $\Omega_i$ of $\Omega$ is convex, then
\begin{equation*}
	m \geq \sum_{i=1}^k \frac 12 \left( \frac{|\Sigma_i|}{\omega_{n-1}} \right) ^{\frac{n-2}{n-1}} + \frac{1}{2(n-1)\omega_{n-1}} \int_{M^n} R \frac{1}{\sqrt{1+|\nabla f|^2}} dV_g.
\end{equation*}
In particular,
\begin{equation*}
	R \geq 0\ \mbox{ implies }\ m \geq \sum_{i=1}^k \frac 12 \left( \frac{|\Sigma_i|}{\omega_{n-1}} \right) ^{\frac{n-2}{n-1}}.
\end{equation*}
\end{corollary}

\begin{remark}
Since
\begin{equation*}
	\sum_{i=1}^k \frac 12 \left( \frac{|\Sigma_i|}{\omega_{n-1}} \right) ^{\frac{n-2}{n-1}} \geq \frac 12 \left( \frac{\sum_{i=1}^k|\Sigma_i|}{\omega_{n-1}} \right) ^{\frac{n-2}{n-1}} =	\frac 12 \left( \frac{A}{\omega_{n-1}} \right) ^{\frac{n-2}{n-1}},
\end{equation*}
Corollary \ref{pic} is a stronger statement than \eqref{rpi} on top of the fact that the lower bound for the ADM mass involves a nonnegative integral when the scalar curvature is nonnegative.
\end{remark}

\section{Positive Mass Theorem for Graphs over $\mathbb R^n$}

Let $(M^n,g) = (\mathbb R^n, \delta+df\otimes df)$ be the graph of a smooth asymptotically flat function $f:\mathbb R^n \to \mathbb R$. Since $g_{ij} = \delta_{ij} + f_if_j$, the inverse of $g_{ij}$ is
	\[g^{ij} = \delta^{ij} - \frac{f^if^j}{1+|\nabla f|^2},\]
where the norm of $\nabla f$ is taken with respect to the flat metric $\delta$ on $\mathbb R^n$. We first compute the Christoffel symbols $\Gamma_{ij}^k$ of $(M^n,g)$:
\begin{align*}
	\Gamma_{ij}^k &= \frac 12 g^{km} (g_{im,j} + g_{jm,i} - g_{ij,m}) \\
	&= \frac 12 \left(\delta^{km} - \frac{f^kf^m}{1+|\nabla f|^2} \right) (f_{ij}f_m + f_if_{jm} + f_{ij}f_m + f_jf_{im} - f_{im}f_j - f_if_{jm}) \\
	&= \frac 12 \left(\delta^{km} - \frac{f^kf^m}{1+|\nabla f|^2} \right) 2f_{ij}f_m \\
	&= f_{ij}f^k - \frac{f_{ij}f^k|\nabla f|^2}{1+|\nabla f|^2} \\
	&= \frac{f_{ij}f^k}{1+|\nabla f|^2}.
\end{align*}

\begin{remark}
	Since the indices are raised and lowered using the flat metric on $\mathbb R^n$, it will be notationally more convenient from now on to write everything as lower indices, with the implicit assumption that any repeated indices are being summed over as usual.
\end{remark}

With the above remark in mind, we have
\begin{align*}
	\Gamma_{ij}^k		&= \frac{f_{ij}f_k}{1+|\nabla f|^2} \\
	\Gamma_{ij,k}^k &= \frac{f_{ijk}f_k}{1+|\nabla f|^2} + \frac{f_{ij}f_{kk}}{1+|\nabla f|^2} - \frac{2f_{ij}f_{kl}f_kf_l}{(1+|\nabla f|^2)^2}.
\end{align*}

We can now compute the scalar curvature $R$ of $(M^n,g)$ using the coordinate expression for scalar curvature:
\begin{align*}
  R &= g^{ij} (\Gamma_{ij,k}^k - \Gamma_{ik,j}^k + \Gamma_{ij}^l \Gamma_{kl}^k - \Gamma_{ik}^l \Gamma_{jl}^k) \\
  &= \left(\delta_{ij} - \frac{f_if_j}{1+|\nabla f|^2} \right) \left( \frac{f_{ijk}f_k}{1+|\nabla f|^2} + \frac{f_{ij}f_{kk}}{1+|\nabla f|^2} - \frac{2f_{ij}f_{kl}f_kf_l}{(1+|\nabla f|^2)^2} - \frac{f_{ijk}f_k}{1+|\nabla f|^2} \right. \\
  &\quad - \frac{f_{ik}f_{jk}}{1+|\nabla f|^2} + \left. \frac{2f_{ik}f_{jl}f_kf_l}{(1+|\nabla f|^2)^2} + \frac{f_{ij}f_{kl}f_kf_l}{(1+|\nabla f|^2)^2} - \frac{f_{ik}f_{jl}f_kf_l}{(1+|\nabla f|^2)^2} \right) \\
  &= \left( \delta_{ij} - \frac{f_if_j}{1+|\nabla f|^2} \right) \left( \frac{f_{ij}f_{kk}}{1+|\nabla f|^2} - \frac{f_{ik}f_{jk}}{1+|\nabla f|^2} - \frac{f_{ij}f_{kl}f_kf_l}{(1+|\nabla f|^2)^2} + \frac{f_{ik}f_{jl}f_kf_l}{(1+|\nabla f|^2)^2} \right) \\
  &= \frac{1}{1+|\nabla f|^2} (f_{ii}f_{kk} - f_{ik}f_{ik}) - \frac{f_kf_l}{(1+|\nabla f|^2)^2} (f_{ii}f_{kl} - f_{ik}f_{il}) \\
  &\quad - \frac{f_if_j}{(1+|\nabla f|^2)^2} (f_{ij}f_{kk} - f_{ik}f_{jk}) - \frac{f_if_jf_kf_l}{(1+|\nabla f|^2)^3}(f_{ij}f_{kl} - f_{ik} f_{jl}). 
\end{align*}

By symmetry, the last term in the last expression is 0. After relabeling the indices, the expression for the scalar curvature is
\begin{equation}\label{R on R^n}
	R = \frac{1}{1+|\nabla f|^2} \left( f_{ii}f_{jj} - f_{ij}f_{ij} - \frac{2f_jf_k}{1+|\nabla f|^2} (f_{ii}f_{jk} - f_{ij}f_{ik}) \right).
\end{equation}

Let us denote by $\nabla\cdot$ the divergence operator on $(\mathbb R^n,\delta)$. The key observation we need to prove Theorem \ref{pmt} is the following lemma:
\begin{lemma} \label{lemma1}
The scalar curvature $R$ of the graph $(\mathbb R^n,\delta+df\otimes df)$ satisfies
\begin{equation*}
  R = \nabla \cdot \left( \frac{1}{1+|\nabla f|^2} (f_{ii}f_j - f_{ij}f_i)\partial_j \right).
\end{equation*}
\end{lemma}

\begin{proof}
This is a direct calculation:
\begin{eqnarray*}
 & & \nabla \cdot \left( \frac{1}{1+|\nabla f|^2} ( f_{ii}f_j - f_{ij}f_i ) \partial_j \right) \\
 &=& \frac{1}{1+|\nabla f|^2} ( f_{iij}f_j + f_{ii}f_{jj} - f_{ijj}f_i - f_{ij}f_{ij} ) - \frac{2f_{jk}f_k}{(1+|\nabla f|^2)^2} (f_{ii}f_j - f_{ij}f_i) \\
 &=& \frac{1}{1+|\nabla f|^2} \left( f_{ii}f_{jj} - f_{ij}f_{ij} - \frac{2f_jf_k}{1+|\nabla f|^2} (f_{ii}f_{jk} - f_{ij}f_{ik}) \right)\\
 &=& R
\end{eqnarray*}
by \eqref{R on R^n}.
\end{proof}

We are now in the position to prove Theorem \ref{pmt}:
\begin{proof}[Proof of Theorem \ref{pmt}]
By definition, the ADM mass of $(M^n,g)=(\mathbb R^n,\delta+df\otimes df)$ is
\begin{align*}
	m &= \lim_{r\to\infty} \frac{1}{2(n-1)\omega_{n-1}} \int_{S_r} (g_{ij,i} - g_{ii,j}) \nu_j dS_r \\
	&= \lim_{r\to\infty} \frac{1}{2(n-1)\omega_{n-1}} \int_{S_r} (f_{ii}f_j + f_{ij}f_i - 2f_{ij}f_i) \nu_j dS_r \\
	&= \lim_{r\to\infty} \frac{1}{2(n-1)\omega_{n-1}} \int_{S_r} (f_{ii}f_j - f_{ij}f_i) \nu_j dS_r.
\end{align*}
By the asymptotic flatness assumption, the function $1/(1+|\nabla f|^2)$ goes to 1 at infinity. Hence we can alternately write the mass as
\begin{equation*}
	m = \lim_{r\to\infty} \frac{1}{2(n-1)\omega_{n-1}} \int_{S_r} \frac{1}{1+|\nabla f|^2} (f_{ii}f_j - f_{ij}f_i) \nu_j dS_r.
\end{equation*}
Now apply the divergence theorem in $(\mathbb R^n,\delta)$ and use Lemma \ref{lemma1} to get
\begin{align*}
	m &= \frac{1}{2(n-1)\omega_{n-1}} \int_{\mathbb R^n} \nabla \cdot \left( \frac{1}{1+|\nabla f|^2} (f_{ii}f_j - f_{ij}f_i)\partial_j \right) dV_{\delta} \\
	&= \frac{1}{2(n-1)\omega_{n-1}} \int_{\mathbb R^n} R dV_{\delta} \\
	&= \frac{1}{2(n-1)\omega_{n-1}} \int_{M^n} R \frac{1}{\sqrt{1+|\nabla f|^2}} dV_g
\end{align*}
since
\begin{equation*}
	dV_g = \sqrt{\det g} dV_{\delta} = \sqrt{1+|\nabla f|^2} dV_{\delta}.
\end{equation*}
\end{proof}

\begin{remark}
	If $(M^n,g)$ is the graph of a smooth {\it spherically symmetric} function $f=f(r)$ on $\mathbb R^n$, then it turns out that the ADM mass of $(M^n,g)$ is nonnegative even without the nonnegative scalar curvature assumption. To see this, let $f_r=\partial f/\partial r$ denote the radial derivative of $f$. By the chain rule, the derivatives of $f$ satisfy
\begin{align*}
	f_i 	 &= f_r\frac{x_i}{r} \\
	f_{ij} &= f_{rr}\frac{x_ix_j}{r^2} + f_r\left(\frac{\delta_{ij}}{r} - \frac{x_ix_j}{r^3}\right).
\end{align*}
The ADM mass of $(M^n,g)$ is
\begin{align*}
	m &= \lim_{r\to\infty} \frac{1}{2(n-1)\omega_{n-1}} \int_{S_r} (g_{ij,i} - g_{ii,j}) \nu_j dS_r \\
		&= \lim_{r\to\infty} \frac{1}{2(n-1)\omega_{n-1}} \int_{S_r} (f_{ii}f_j - f_{ij}f_i) \nu_j dS_r \\
		&= \lim_{r\to\infty} \frac{1}{2(n-1)\omega_{n-1}} \int_{S_r} f_{rr} \frac{x_i^2x_j^2}{r^2} + f_r^2 \left( \frac{x_j^2}{r^3} - \frac{x_i^2x_j^2}{r^5} \right) - f_{rr}\frac{x_i^2x_j^2}{r^2} \\
		&\qquad - f_r^2 \left( \frac{\delta_{ij}x_ix_j}{r^3} - \frac{x_i^2x_j^2}{r^5} \right) dS_r \\
		&= \lim_{r\to\infty} \frac{1}{2(n-1)\omega_{n-1}} \int_{S_r} \frac{2f_r^2}{r}\geq 0.
\end{align*}
A consequence of this fact and Theorem \ref{pmt} is that there are no spherically symmetric asymptotically flat smooth functions on $\mathbb R^n$ whose graphs have negative scalar curvature everywhere.
\end{remark}

\section{Penrose Inequality for Graphs over $\mathbb R^n$}

Let $\Omega$ be a bounded open set in $\mathbb R^n$ and $\Sigma = \partial \Omega$. If $f: \mathbb R^n \backslash \Omega \to \mathbb R$ is a smooth asymptotically flat function such that each connected component of $f(\Sigma)$ is in a level of $f$ and $|\nabla f(x)|\to\infty$ as $x\to\Sigma$, then the graph of $f$, $(M^n,g) = (\mathbb R^n \backslash \Omega, \delta + df\otimes df)$, is an asymptotically flat manifold with area outer minimizing horizon $\Sigma$. In this setting, proving Theorem \ref{pi} is a matter of keeping track of the extra boundary term when we apply the divergence theorem in the proof of Theorem \ref{pmt}. As before, any repeated indices are being summed over.

\begin{proof}[Proof of Theorem \ref{pi}]
	As in the proof of Theorem \ref{pmt}, we can write the mass of $(M^n,g)$ as
\begin{equation*}
	m = \lim_{r\to\infty} \frac{1}{2(n-1)\omega_{n-1}} \int_{S_r} \frac{1}{1+|\nabla f|^2} (f_{ii}f_j - f_{ij}f_i) \nu_j dS_r. 
\end{equation*}
The difference here is that when we apply the divergence theorem, we get an extra boundary integral: 
\begin{align*}
	m &= \lim_{r\to\infty} \frac{1}{2(n-1)\omega_{n-1}} \int_{S_r} \frac{1}{1+|\nabla f|^2} (f_{ii}f_j - f_{ij}f_i) \nu_j dS_r \\
	&= \frac{1}{2(n-1)\omega_{n-1}} \int_{\mathbb R^n\backslash\Omega} \nabla \cdot \left( \frac{1}{1+|\nabla f|^2} (f_{ii}f_j - f_{ij}f_i)\partial_j \right) dV_{\delta} \\
	&\quad - \frac{1}{2(n-1)\omega_{n-1}} \int_{\Sigma} \frac{1}{1+|\nabla f|^2} (f_{ii}f_j - f_{ij}f_i) \nu_j d\Sigma \\
	&= \frac{1}{2(n-1)\omega_{n-1}} \int_{M^n} R\frac{1}{\sqrt{1+|\nabla f|^2}} dV_g \\
	&\quad - \frac{1}{2(n-1)\omega_{n-1}} \int_{\Sigma} \frac{1}{1+|\nabla f|^2} (f_{ii}f_j - f_{ij}f_i) \nu_j d\Sigma.
\end{align*}
The outward normal to $\Sigma$ is $\nu=-\nabla f/ |\nabla f|$. Let $\Delta f$ be the Laplacian of $f$ in $(M^n,g)$ and $\Delta_{\Sigma}f$ the Laplacian of $f$ along $\Sigma$. Let $H^f$ denote the Hessian of $f$ and $H_0$ the mean curvature of $\Sigma$ with respect to the flat metric. We will use the following well known formula to relate the two Laplacians:
\begin{align*}
	\Delta f &= \Delta_{\Sigma}f + H^f(\nu,\nu) + H_0 \cdot \nu(f) \\
	&= \frac{1}{|\nabla f|} H^f \left( \nabla f, \frac{\nabla f}{|\nabla f|} \right) + H_0 |\nabla f|,
\end{align*}
where $\Delta_{\Sigma}f=0$ since $f$ is constant on $\Sigma$. Now
\begin{align*}
	& -\frac{1}{1+|\nabla f|^2} (f_{ii}f_j - f_{ij}f_i) \nu_j \\
	&= \frac{1}{1+|\nabla f|^2} \left[ (\Delta f) |\nabla f| - H^f \left( \nabla f, \frac{\nabla f}{|\nabla f|} \right) \right] \\
	&= \frac{1}{1+|\nabla f|^2} \left[ \left( \frac{1}{|\nabla f|} H^f \left( \nabla f, \frac{\nabla f}{|\nabla f|} \right) + H_0 |\nabla f| \right) |\nabla f| - H^f \left( \nabla f, \frac{\nabla f}{|\nabla f|} \right) \right] \\
	&= \frac{|\nabla f|^2}{1+|\nabla f|^2} H_0.
\end{align*}
Therefore,
\begin{align*}
	m &= \frac{1}{2(n-1)\omega_{n-1}} \int_{M^n} R\frac{1}{\sqrt{1+|\nabla f|^2}} dV_g + \frac{1}{2(n-1)\omega_{n-1}} \int_{\Sigma} \frac{|\nabla f|^2}{1+|\nabla f|^2} H_0 d\Sigma \\
	&= \frac{1}{2(n-1)\omega_{n-1}} \int_{M^n} R\frac{1}{\sqrt{1+|\nabla f|^2}} dV_g + \frac{1}{2(n-1)\omega_{n-1}} \int_{\Sigma} H_0 d\Sigma.
\end{align*}
\end{proof}

Let us denote by $\Omega_i$, $i=1,\ldots,k$ the connected components of the bounded open set $\Omega$. In the case that each $\Omega_i$ is convex, it turns out we can obtain a lower bound for the boundary integral in Theorem \ref{pi}. To do this, we will need the following lemma, which is a special case of the Aleksandrov-Fenchel inequality \cite{schneider}.
\begin{lemma} \label{lemma2}
If $\Sigma$ is a convex surface in $\mathbb R^n$ with mean curvature $H_0$ and area $|\Sigma|$, then
\begin{equation*}
	\frac{1}{2(n-1)\omega_{n-1}} \int_{\Sigma} H_0 \geq \frac 12 \left( \frac{|\Sigma|}{\omega_{n-1}} \right) ^{\frac{n-2}{n-1}}.
\end{equation*}
\end{lemma}

\begin{proof}
Let $\Sigma \subset \mathbb R^n$ be a convex surface with principal curvatures $\kappa_1, \ldots, \kappa_{n-1}$. Let
\begin{equation*}
	\sigma_j (\kappa_1, \ldots, \kappa_{n-1}) = \binom{n-1}{j}^{-1} \sum_{1\leq i_1<\cdots<i_k\leq n-1} \kappa_{i_1} \cdots \kappa_{i_j}
\end{equation*}
be the $j$th normalized elementary symmetric functions in $\kappa_1, \ldots, \kappa_{n-1}$ for $j=1,\ldots, n-1$. In particular,
\begin{align*}
	\sigma_0 (\kappa_1, \ldots, \kappa_{n-1}) &= 1 \\
	\sigma_1 (\kappa_1, \ldots, \kappa_{n-1}) &= \frac{1}{n-1} \sum_{i=1}^{n-1} \kappa_i = \frac{1}{n-1} H_0\\
	\sigma_{n-1} (\kappa_1, \ldots, \kappa_{n-1}) &= \prod_{i=1}^{n-1} \kappa_i.
\end{align*}
The $k$th quermassintegral $V_k$ of $\Sigma$ is defined to be
\begin{equation*}
	V_k = \int_{\Sigma} \sigma_k (\kappa_1, \ldots, \kappa_{n-1}).
\end{equation*}
A special case of the Aleksandrov-Fenchel inequality states that, for $0\leq i<j<k\leq n-1$,
\begin{equation*}
	V_j^{k-1} \geq V_i^{k-j} V_k^{j-i}.
\end{equation*}
Taking $i=0$, $j=1$, $k=n-1$,
\begin{equation}\label{af}
	V_1^{n-1} \geq V_0^{n-2} V_{n-1}.
\end{equation}
Now
\begin{align*}
	V_0 &= \int_{\Sigma} \sigma_0 (\kappa_1, \ldots, \kappa_{n-1}) = |\Sigma|  \\
	V_1 &= \int_{\Sigma} \sigma_1 (\kappa_1, \ldots, \kappa_{n-1}) = \frac{1}{n-1} \int_{\Sigma} H_0 \\
	V_{n-1} &= \int_{\Sigma} \sigma_{n-1} (\kappa_1, \ldots, \kappa_{n-1}) = \omega_{n-1}.  
\end{align*}
Thus \eqref{af} becomes
\begin{align*}
	\left( \frac{1}{n-1} \int_{\Sigma} H_0 \right) ^{n-1} &\geq |\Sigma|^{n-2} \omega_{n-1} \\
	\frac{1}{n-1} \int_{\Sigma} H_0 &\geq |\Sigma| ^{\frac{n-2}{n-1}} \omega_{n-1} ^{\frac{1}{n-1}} \\
	\frac{1}{2(n-1)\omega_{n-1}} \int_{\Sigma} H_0 &\geq \frac 12 \left( \frac{|\Sigma|}{\omega_{n-1}} \right) ^{\frac{n-2}{n-1}}
\end{align*}
as claimed.
\end{proof}
Now Corollary \ref{pic} follows directly from Theorem \ref{pi} and Lemma \ref{lemma2}.

\section{Discussions and Acknowledgements}

We point out that the technique in this paper can be used to study the mass of an asymptotically flat manifold that can be isometrically embedded as a graph in Minkowski space, and a second paper in this direction is currently under progress. In this setting, zero area singularities can arise. The standard sample of such manifolds is the Schwarzschild metric with $m<0$. For more on this topic, we refer the readers to \cites{bray_npms, bray_jauregui}.

We also point out that recently Schwartz \cite{schwartz} proved a volumetric Penrose inequality for conformally flat manifolds in all dimensions $n\geq 3$. 

I would like to thank my advisor Hubert Bray for suggesting this problem and for his supervision and guidance. I would also like to thank Ben Andrews for pointing out the Aleksandrov-Fenchel inequality. Finally, I would like to thank Graham Cox and Jeffrey Jauregui for many helpful discussions.

\end{document}